\input xy
\xyoption{all}\xyoption{poly}
\newdir{ >}{{}*!/-9pt/@{>}}

\documentclass[reqno]{amsart}
\usepackage{amsmath,latexsym,amssymb,mathrsfs}

\usepackage{textcomp}
\usepackage{graphicx}

\usepackage{hyperref}

\theoremstyle{plain}
\newtheorem{theorem}{Theorem}[section]
\newtheorem{lemma}[theorem]{Lemma}

\newtheorem{corollary}[theorem]{Corollary}

\theoremstyle{definition}
\newtheorem{definition}[theorem]{Definition}

\newtheorem{example}[theorem]{Example}

\theoremstyle{remark}
\newtheorem{remark}[theorem]{Remark}

\begin{document}
\title[Star-regularity and regular completions]{Star-regularity and regular completions}

\author{Marino Gran }

\author{Zurab Janelidze}
\address{
\noindent  Institut de Recherche en Math\'ematique et Physique, Universit\'e catholique de Louvain, Chemin du Cyclotron 2, 1348 Louvain-la-Neuve, Belgium,  \, and \, 
Department of Mathematical Sciences, Stellenbosch University, Private Bag X1, Matieland, 7602, South Africa}
\address{
Department of Mathematical Sciences, Stellenbosch University, Private Bag X1, Matieland, 7602, South Africa}

\begin{abstract} In this paper we establish a new characterisation of star-regular categories, using a property of internal reflexive graphs, which is suggested by a recent result due to O.~Ngaha~Ngaha and the first author. We show that this property is, in a suitable sense, invariant under regular completion of a category in the sense of A.~Carboni and E.~M.~Vitale. Restricting to pointed categories, where star-regularity becomes normality in the sense of the second author, this reveals an unusual behaviour of the exactness property of normality (i.e.~the property that regular epimorphisms are normal epimorphisms) compared to other closely related exactness properties studied in categorical algebra. \\

\noindent Math. Subj. Class. 2010: 18A35, 18C05, 18D20, 08B20, 08B30, 08B05.
\end{abstract}
\maketitle

\section{Introduction}

The context of a \emph{star-regular category} was introduced in \cite{GraZJanUrs12} (see also \cite{GraZJanRodUrs12,GraZJanRod12}) for a unification of some parallel results in the study of exactness properties in the context of regular categories, and exactness properties in the context of \emph{normal categories} \cite{ZJan10}, i.e.~pointed regular categories where every regular epimorphism is a normal epimorphism.

A star-regular category is a regular category equipped with an \emph{ideal} in the sense of C.~Ehresmann \cite{Ehr64} (i.e.~a class of morphisms which is closed under both left and right composition with morphisms in the category), satisfying suitable conditions (we recall the precise definition in Section~\ref{Equ C} below). In particular, when the ideal consists of the zero morphisms in a pointed regular category, star-regularity states that every regular epimorphism is a normal epimorphism. In the universal-algebraic language, this can be equivalently reformulated by saying that any congruence is ``generated'' by its $0$-class. Thus, for pointed varieties this is precisely the well-known \emph{$0$-regularity}, which was first studied (under a different name) in \cite{Fic68}. 

From the syntactic characterisation of $0$-regular varieties obtained in \cite{Fic68}, the following property of such varieties can be easily deduced: the congruence generated by the $0$-class of a reflexive homomorphic relation always contains the reflexive relation as a subrelation. In \cite{GraNga13} it was shown that a similar fact holds more generally in any star-regular category. In this paper we refine and further generalise this result --- see Theorem~\ref{The A} below. This leads to showing that star-regular categories $\mathbb{C}$ which appear as a regular completion (in the sense of \cite{Vit94, CarVit98}) of a category $\mathbb{P}$ with weak finite limits, can be characterised using a property of reflexive graphs in $\mathbb{P}$, expressed in terms of the ideal of $\mathbb{C}$ restricted to $\mathbb{P}$ (see Theorem~\ref{The C} below). Finally, at the end of the paper, we explicitly consider the pointed case, where star-regularity becomes normality. We point out some differences between how the regular completion behaves with respect to normality, and how it behaves with respect to some of the other closely related exactness properties, as established in \cite{RosVit01,Gra02,GraRod12}.

\section{Internal reflexive graphs}\label{Equ C}

As briefly recalled in the Introduction, an \emph{ideal} in a category $\mathbb{C}$ is a class $\mathcal{N}$ of morphisms of $\mathbb{C}$ such that for any composite $fgh$ of three morphisms
$$\xymatrix{ W\ar[r]^-{h}\ar@/_10pt/[rrr]_-{fgh} & X\ar[r]^-{g} & Y\ar[r]^-{f} & Z}$$
if $g\in\mathcal{N}$ then $fgh\in\mathcal{N}$. An ideal can be equivalently seen as a subfunctor of the hom-functor, and hence also as an enrichment in the category $\mathsf{Set}_2$ of \emph{pairs of sets}, i.e.~the category arising from the Grothendieck construction applied to the covariant powerset functor $P:\mathsf{Set}\rightarrow\mathsf{Cat}$ (which is equivalent to the category of monomorphisms in $\mathsf{Set}$). In \cite{GraZJanUrs12}, objects of $\mathsf{Set}_2$, which are pairs $(A,B)$ where $A$ is a set and $B$ is a subset of $A$, were called \emph{multi-pointed sets}, since we can think of the subset $B$ as a \emph{set of base points} of $A$, in analogy with \emph{pointed sets} where there is always only one base point. Then, a pair $(\mathbb{C},\mathcal{N})$ where $\mathbb{C}$ is a category and $\mathcal{N}$ is an ideal of $\mathbb{C}$ is called a \emph{multi-pointed category}. Pairs of sets, in the above sense, occur naturally in homological treatment of algebraic topology which in turn motivates the study of multi-pointed categories, as it has been thoroughly clarified by the work of M.~Grandis on a non-additive generalisation of homological algebra \cite{MGra12,MGra13}. In this theory, which was first elaborated in \cite{MGra92}, multi-pointed categories enter as a natural generalisation of pointed categories. This aspect is still central in our work on ``stars'' \cite{GraZJanUrs12,GraZJanRodUrs12,GraZJanRod12}, but, at the same time, we think of multi-pointed categories \emph{also} as generalised ordinary (not necessarily pointed) categories. An ordinary category can be seen as a multi-pointed category where the ideal consists of all morphisms --- we call this the \emph{total context}. This gives only trivial examples in the realm of \cite{MGra13}, while for what we do in a multi-pointed category, the total context is as important as the \emph{pointed context} where every hom-set contains exactly one morphism from the ideal, in which case the ideal is uniquely determined and a multi-pointed category becomes a pointed category (see Remark~\ref{RemA} below).

Thinking of a multi-pointed category as a generalisation of a pointed category, in it one defines a kernel $k:K\rightarrow X$ of a morphism $f:X\rightarrow Y$ as a morphism such that $fk$ belongs to the given ideal $\mathcal{N}$, and universal with this property: if $fk'$ also belongs to $\mathcal{N}$ for some morphism $k':K'\rightarrow X$, then $k'=ku$ for a unique morphism $u:K'\rightarrow K$, as illustrated in the following display:
$$\xymatrix{ K\ar[r]^-{k} & X\ar[r]^-{f} & Y \\ K'\ar[ur]_-{k'}\ar@{..>}[u]^-{\exists! u} & & }$$
To avoid confusion when we work with several ideals, we will say \emph{$\mathcal{N}$-kernel} instead of \emph{kernel}. This notion is one of the basic notions in the above-mentioned work of M.~Grandis. In our work on multi-pointed categories, a central role is played by the following notion: given an ordered pair 
\begin{equation}\label{Equ A} \xymatrix{ X\ar@/^3pt/[r]^-{f_1}\ar@/_3pt/[r]_-{f_2} & Y }\end{equation}
of parallel morphisms, its \emph{star} is defined as an ordered pair $(f_1k,f_2k)$ where $k$ is a kernel of $f_1$. Again, when the ideal $\mathcal{N}$ is not fixed, one should replace here ``star'' and ``kernel'' with ``$\mathcal{N}$-star'' and ``$\mathcal{N}$-kernel'', respectively. The transition from a pair to its star will be usually indicated by the display 
$$\xymatrix{ K\ar[r]^-{k} & X\ar@/^3pt/[r]^-{f_1}\ar@/_3pt/[r]_-{f_2} & Y }$$
For the purposes of the present paper, we will also need the notion of a \emph{weak kernel} which is defined in the same way as the notion of a kernel, but dropping the uniqueness requirement (similarly as one usually defines weak limits). Then a \emph{weak star} of a pair $(f_1,f_2)$ is a pair $(f_1k,f_2k)$ where $k$ is a weak kernel of $f_1$. The formal definition is the following:

\begin{definition}
Let $\mathbb{C}$ be a category and let $\mathcal{N}$ be an ideal in $\mathbb{C}$. A \emph{weak kernel} (or more precisely, a \emph{weak $\mathcal{N}$-kernel}) of a morphism $f:X\rightarrow Y$ is a morphism $k:K\rightarrow X$ such that $fk$ belongs to $\mathcal{N}$, and weakly universal with this property: if $fk'$ also belongs to $\mathcal{N}$ for some morphism $k':K'\rightarrow X$, then $k'=ku$ for some morphism $u:K'\rightarrow K$ (which is not required to be unique):
$$\xymatrix{ K\ar[r]^-{k} & X\ar[r]^-{f} & Y \\ K'\ar[ur]_-{k'}\ar@{..>}[u]^-{\exists u} & & }$$
A \emph{weak star} (or more precisely, a \emph{weak $\mathcal{N}$-star}) of a pair $(f_1,f_2)$ of parallel morphisms in $\mathbb{C}$ is defined as a pair $(f_1k,f_2k)$ where $k$ is a weak $\mathcal{N}$-kernel of $f_1$.
\end{definition}

As explained in \cite{GraZJanUrs12}, the term ``star'' goes back to \cite{Bro70,Bro06} and its subsequent use, closer to our considerations, in \cite{GJan03}. 

Note that a star of a pair is at the same time a weak star, since a kernel of a morphism is at the same time a weak kernel.

\subsection*{The condition $(\ast\pi_0)$.}
Let $\mathbb{C}$ be a category and let $\mathcal{N}$ be an ideal in $\mathbb{C}$.
Consider the following condition on an ordered pair $$\xymatrix{ X\ar@/^3pt/[r]^-{f_1}\ar@/_3pt/[r]_-{f_2} & Y }$$ of parallel morphisms:
\begin{itemize}
\item[$(\ast\pi_0)$] For a weak $\mathcal{N}$-star $(f_1k,f_2k)$ of $(f_1,f_2)$, and for any morphism $g:Y\rightarrow Z$,
$$
gf_1k=gf_2k\quad\Leftrightarrow\quad gf_1=gf_2. $$
\end{itemize}

Note that when a pair does have at least one weak star, then in the condition $(\ast\pi_0)$ above ``any of its weak $\mathcal{N}$-stars'' can be equivalently replaced with ``one of its weak $\mathcal{N}$-stars''. Moreover, when the pair even has a star, then we can further equivalently replace ``any of its weak $\mathcal{N}$-stars'' in the above condition with ``any of its $\mathcal{N}$-stars'' or ``one of its $\mathcal{N}$-stars''. 

The condition $(\ast\pi_0)$ can be seen as a way of expressing that the connected components of the pair $(f_1,f_2)$, seen as an internal graph, are entirely determined by its ``star'', the subgraph $(f_1k, f_2k)$. Here ``$\ast$'' stands for ``star'' and ``$\pi_0$'' is the same as in the notation ``$\pi_0(G)$'' for the set of connected components of a graph $G$ --- so, ``$(\ast\pi_0)$'' can be read as: \emph{star determines connected components}. When either the pair $(f_1,f_2)$ or the pair $(f_1k, f_2k)$ has a coequalizer, the condition $(\ast\pi_0)$ is equivalent to coequalizers of $(f_1,f_2)$ and $(f_1k, f_2k)$ being the same.  

In \cite{GraZJanUrs12}, a \emph{star-regular category} was defined as a pair $(\mathbb{C},\mathcal{N})$ where $\mathbb{C}$ is a regular category, $\mathcal{N}$ is an ideal in $\mathbb{C}$ admitting kernels, and every regular epimorphism in $\mathbb{C}$ is a coequalizer of its \emph{kernel star}, i.e.~the star of its kernel pair. So, since a regular epimorphism is always a coequalizer of its kernel pair, and since any kernel pair in a regular category has a coequalizer, star-regularity states precisely that kernel pairs satisfy $(\ast\pi_0)$. In \cite{GraNga13} it was shown that in a star-regular category every reflexive relation satisfies $(\ast\pi_0)$. In fact, we have:

\begin{theorem}\label{The A}
In a multi-pointed category $(\mathbb{C},\mathcal{N})$ with weak kernels and weak kernel pairs, the following conditions are equivalent:
\begin{enumerate}
\item Given a reflexive graph 
$$\xymatrix{ G_1\ar@/^4pt/[r]^-{d}\ar@/_4pt/[r]_-{c} & G_0\ar[l]|-{e}, }$$
 its underlying graph, regarded as a pair $(d,c)$, satisfies $(\ast\pi_0)$.

\item Every weak kernel pair satisfies $(\ast\pi_0)$.
\end{enumerate} 
Moreover, when kernel pairs exist in $\mathbb C$, these conditions are also equivalent to the following ones:
\begin{enumerate}\setcounter{enumi}{2}
\item Every reflexive relation satisfies $(\ast\pi_0)$.
\item Every kernel pair satisfies $(\ast\pi_0)$.
\end{enumerate} 
\end{theorem}

\begin{proof}
A weak kernel pair is, in particular, a reflexive graph, so (a)$\Rightarrow$(b) is obvious. For (b)$\Rightarrow$(a) we adapt the argument given in \cite{GraNga13} to the present context: consider the display
\begin{equation}\label{Equ B}\xymatrix{ K\ar[r]^-{k} & G_1\ar@/^4pt/[r]^-{d}\ar@/_4pt/[r]_-{c} & G_0\ar[l]|-{e}\ar[r]^-{l} & L }\end{equation}
where $k$ is a weak kernel of $d$, while $de=ce=1_{G_0}$ and $l$ is any morphism. We must show
$$ldk=lck\quad\Rightarrow\quad ld=lc$$
(for, the converse implication is obvious). So we suppose $ldk=lck$. Consider the following similar chain
$$\xymatrix{ K'\ar[r]^-{k'} & P\ar@/^4pt/[r]^-{u}\ar@/_4pt/[r]_-{v} & G_1\ar[l]|-{i}\ar[r]^-{d} & G_0 }$$
where $(u,v)$ is a weak kernel pair of $d$ and hence admits a morphism $i$ as displayed above, such that $ui=vi=1_{G_1}$; similarly as before, $k'$ is a weak kernel of $u$. Now, we have $$dvk'=duk'\in\mathcal{N}.$$
This gives rise to morphisms $u',v'$ as in the display
$$\xymatrix@=50pt{ K'\ar@<+2pt>[rd]^-{uk'}\ar@<-2pt>[rd]_-{vk'}\ar@<+2pt>[d]^-{u'}\ar@<-2pt>[d]_-{v'} & & & \\ K\ar[r]^-{k} & G_1\ar@/^7pt/[r]^-{d}\ar@/_7pt/[r]_-{c} & G_0\ar[l]|-{e}\ar[r]^-{l} & L }$$
with $kv'=vk'$ and $ku'=uk'$. To be able to deduce $ld=lc$ we need one more arrow. Since $ded=d$, we get an induced morphism $s$ into the domain of the weak kernel pair $(v, u)$ of $d$, as in the display
$$\xymatrix{ G_1\ar@/_5pt/[ddr]_-{ed}\ar@/^5pt/[drr]^-{1_{G_1}}\ar[dr]|-{s} & & \\ & P\ar[d]_-{v}\ar[r]^-{u} & G_1\ar[d]^-{d} \\ & G_1\ar[r]_-{d} & G_0 }$$
This morphism, together with $e$ and $i$, shows that $d$ is a split coequalizer of its weak kernel pair $(u,v)$. In particular, this implies that if we can show $lcu=lcv$ then we will have $lc=ld$ --- indeed it does, as then $lc=lcus=lcvs=lced=ld$. Since the pair $(u,v)$ satisfies $(\ast\pi_0)$ by assumption, to get $lcu=lcv$ it is sufficient to show $lcuk'=lcvk'$. Indeed,
$$ lcuk' = lcku' = ldku' = lduk' = ldvk' = ldkv' = lckv' = lcvk'.$$
This shows (b)$\Rightarrow$(a). For the second part of the theorem, suppose that $\mathbb{C}$ has kernel pairs. We have obvious implications (a)$\Rightarrow$(c)$\Rightarrow$(d). The implication (d)$\Rightarrow$(b) follows easily from the fact that the canonical morphism from a weak kernel pair of some morphism in $\mathbb C$ to the kernel pair of the same morphism is a split epimorphism.
\end{proof}

In the total context, i.e.~when $\mathcal{N}$ is the class of all morphisms in a category $\mathbb{C}$, any pair trivially satisfies $(\ast\pi_0)$, and so all conditions in Theorem~\ref{The A} hold trivially (thus in this case the assumption that weak or strict kernel pairs must exist is redundant). In the pointed context, when in addition $\mathbb{C}$ is a regular category, the result above gives a new characterisation of normal categories in the sense of \cite{ZJan10}. More generally, we have:

\begin{corollary}\label{Cor A}
For a multi-pointed category $(\mathbb{C},\mathcal{N})$ having kernel pairs and kernels, and having coequalizers of kernel pairs, every regular epimorphism is a coequalizer of its kernel star if and only if every reflexive graph in $\mathbb{C}$ satisfies $(\ast\pi_0)$. In particular, this gives that a regular multi-pointed category is a star-regular category if and only if every internal reflexive graph in it satisfies $(\ast\pi_0)$.
\end{corollary}

\begin{proof}
This follows from Theorem~\ref{The A} and the fact that in a multi-pointed category, under the existence of kernel pairs, kernels, and coequalizers of kernel pairs, \ref{The A}(d) is equivalent to every regular epimorphism being a coequalizer of its kernel star.  
\end{proof}

Similarly, we have:

\begin{corollary}\label{Cor D}
For a multi-pointed category $(\mathbb{P},\mathcal{N})$ having weak kernel pairs and weak kernels, and having coequalizers of weak kernel pairs, every regular epimorphism is a coequalizer of its weak kernel star (i.e.~weak star of its weak kernel pair) if and only if every reflexive graph in $\mathbb{P}$ satisfies $(\ast\pi_0)$.
\end{corollary}

\section{Ideals and projective covers}\label{SecA}

Recall from \cite{CarVit98} that a \emph{projective cover} $\mathbb{P}$ of a regular category $\mathbb{C}$ is a full subcategory $\mathbb{P}$ of $\mathbb{C}$ such that:
\begin{itemize}
\item Every object $P$ in $\mathbb{P}$ is (regular) projective in $\mathbb{C}$, i.e.~the functor $$\mathsf{hom}(P,-):\mathbb{C}\rightarrow\mathbf{Set}$$ preserves regular epimorphisms;

\item For every object $C$ in $\mathbb{C}$ there exists a regular epimorphism $P\rightarrow C$ where $P\in\mathbb{P}$.
\end{itemize}
Below we study the question of extending ideals of morphisms from $\mathbb{P}$ to $\mathbb{C}$. Ideals on $\mathbb{C}$ arising in this way will have a simple characteristic property. 

Let $\mathbb{C}$ be a regular category and let $\mathcal{N}$ be an ideal in $\mathbb{C}$. A morphism $f:P\rightarrow Y$ in $\mathbb{C}$ is said to be \emph{$\mathcal{N}$-saturating} (or simply \emph{saturating}, when the ideal has been fixed)  when for any morphism $n:X\rightarrow Y$ from $\mathcal{N}$ there is a commutative square
$$\xymatrix{ Z\ar@{->>}[d]_-{g}\ar[r]^-{m} & P\ar[d]^-{f} \\ X\ar[r]_-{n} & Y }$$  
with $g$ a regular epimorphism and $m\in\mathcal{N}$; we say in this case that $f$ \emph{saturates} $n$. When $\mathcal{N}$ admits kernels, this becomes the same notion as the one used in \cite{GraZJanRodUrs12}. In the total context saturating morphisms are precisely the regular epimorphisms, while in the pointed context all morphisms are saturating. 

Let $\mathbb{C}$ be a regular category and let $\mathbb{P}$ be a projective cover of $\mathbb{C}$. Let $\mathfrak{I}(\mathbb{C})$ and $\mathfrak{I}(\mathbb{P})$ denote the ordered sets of ideals (ordered under inclusion of ideals) of $\mathbb{C}$ and $\mathbb{P}$, respectively. Note that ideals are always closed under arbitrary unions and intersections, and so both $\mathfrak{I}(\mathbb{C})$ and $\mathfrak{I}(\mathbb{P})$ are complete lattices. In general, there are several order-preserving maps between them. In particular, consider the maps
$$\mathfrak{I}(\mathbb{C})\rightarrow \mathfrak{I}(\mathbb{P}),\quad \mathcal{N}\mapsto\mathcal{N}_\mathbb{P}\quad\textrm{(restriction map)}$$
$$\mathfrak{I}(\mathbb{P})\rightarrow \mathfrak{I}(\mathbb{C}),\quad \mathcal{N}\mapsto\mathcal{N}^\mathbb{C}\quad\textrm{(extension map)}$$
where
\begin{itemize}
\item for an ideal $\mathcal{N}$ of morphisms in $\mathbb{C}$, by $\mathcal{N}_\mathbb{P}$ we denote the restriction of $\mathcal{N}$ on $\mathbb{P}$, i.e.~
$$\mathcal{N}_\mathbb{P}=\mathcal{N}\cap\mathbb{P}_1$$
where $\mathbb{P}_1$ denotes the class of all morphisms in $\mathbb{P}$; 

\item for an ideal $\mathcal{N}$ of morphisms in $\mathbb{P}$, by $\mathcal{N}^\mathbb{C}$ we denote the class of morphisms in $\mathbb{C}$ consisting of those morphisms $f:X\rightarrow Y$ in  $\mathbb{C}$ for which there is a commutative square
$$\xymatrix{ P\ar[r]^-{n}\ar@{->>}[d]_-{e} & P'\ar@{->>}[d]^-{e'} \\ X\ar[r]_-{f} & Y }$$ where $e$ and $e'$ are regular epimorphisms and $n\in\mathcal{N}$.
\end{itemize}
The restriction map preserves arbitrary joins and meets, and so has both left and right adjoints. The extension map preserves arbitrary joins and so has a right adjoint. In general the restriction and the extension maps are not adjoints to each other, as the next example shows.
\begin{example}
Let $\mathbb{C}=\mathbf{Grp}$ be the category of groups and let $\mathbb{P}$ be its projective cover whose objects are the free groups. Consider the ideal $\mathcal{N}$ in  $\mathbf{Grp}$ consisting of those morphisms which either factor through the group $\mathbb{Z}$ of integers, or through $\mathbb{Z}_2+\mathbb{Z}_2$ (where $\mathbb{Z}_2$ is the group $\mathbb{Z}/2\mathbb{Z}$). Then, 
\begin{itemize}
\item the identity morphism 
$\mathbb{Z}_3\rightarrow\mathbb{Z}_3$ belongs to $(\mathcal{N}_\mathbb{P})^\mathbb{C}$ since it is part of the commutative diagram
$$\xymatrix{ \mathbb{Z}\ar[r]^-{1_\mathbb{Z}}\ar@{->>}[d] & \mathbb{Z}\ar@{->>}[d] \\ \mathbb{Z}_3\ar[r]_-{1_{\mathbb{Z}_3}} & \mathbb{Z}_3 }$$
where the top arrow belongs to $\mathcal{N}_\mathbb{P}$, but it cannot belong to $\mathcal{N}$ since the only morphism $\mathbb{Z}_3\rightarrow\mathbb{Z}$ is the trivial one, and at the same time the only morphism $\mathbb{Z}_2+\mathbb{Z}_2\rightarrow\mathbb{Z}_3$ is also the trivial one;

\item the identity morphism $\mathbb{Z}_2+\mathbb{Z}_2\rightarrow\mathbb{Z}_2+\mathbb{Z}_2$ belongs to $\mathcal{N}$, but it does not belong to $(\mathcal{N}_\mathbb{P})^\mathbb{C}$ since in any commutative diagram
$$\xymatrix{ G\ar[r]^-{n}\ar@{->>}[d] & H\ar@{->>}[d] \\ \mathbb{Z}_2+\mathbb{Z}_2\ar[r]_-{1_{\mathbb{Z}_2+\mathbb{Z}_2}} & \mathbb{Z}_2+\mathbb{Z}_2 }$$
where $G$ and $H$ are free groups, the morphism $n$ can neither factor through $\mathbb{Z}$ (if it did, then we would get that there is a surjective homomorphism from $\mathbb{Z}$ to $\mathbb{Z}_2+\mathbb{Z}_2$ which is clearly false), and nor it can factor through $\mathbb{Z}_2+\mathbb{Z}_2$ since any such morphism would be a trivial homomorphism between the free groups (and then it could not make the diagram above commute since the bottom arrow is not trivial and the left vertical arrow is a surjective homomorphism).  
\end{itemize}
Thus, in general, it is neither true that $(\mathcal{N}_\mathbb{P})^\mathbb{C}\subseteq\mathcal{N}$ nor that $\mathcal{N}\subseteq(\mathcal{N}_\mathbb{P})^\mathbb{C}$: this implies that the two constructions $\mathcal{N}_\mathbb{P}$ and $\mathcal{N}^\mathbb{C}$ cannot be adjoint to each other. 
\end{example}

\begin{lemma}\label{Lem A} For any regular category $\mathbb{C}$ having a projective cover $\mathbb{P}$ we have:
\begin{enumerate}
\item $\mathcal{N}=(\mathcal{N}^\mathbb{C})_\mathbb{P}$ for any ideal $\mathcal{N}$ in $\mathbb{P}$.

\item For any ideal $\mathcal{N}$ in $\mathbb{P}$, the category $\mathbb{P}$ has weak $\mathcal{N}$-kernels if and only if $\mathbb{C}$ has $\mathcal{N}^\mathbb{C}$-kernels.

\item For any ideal $\mathcal{N}$ in $\mathbb{C}$, if $\mathbb{C}$ has $\mathcal{N}$-kernels then $\mathbb{P}$ has weak $\mathcal{N}_\mathbb{P}$-kernels and, moreover, in this case we have $(\mathcal{N}_\mathbb{P})^\mathbb{C}\subseteq\mathcal{N}$.

\item For any ideal $\mathcal{N}$ in $\mathbb{C}$, we have $\mathcal{N}\subseteq(\mathcal{N}_\mathbb{P})^\mathbb{C}$ if and only if regular epimorphisms are $\mathcal{N}$-saturating in $\mathbb{C}$.

\item For any ideal $\mathcal{N}$ in $\mathbb{P}$, regular epimorphisms are $\mathcal{N}^\mathbb{C}$-saturating in $\mathbb{C}$.
\end{enumerate} 
\end{lemma}

\begin{proof}
(a): It is easy to see that $\mathcal{N}\subseteq\mathcal{N}^\mathbb{C}$. Since obviously $\mathcal{N}\subseteq\mathbb{P}_1$ (where $\mathbb{P}_1$ is the class of all morphisms in $\mathbb{P}$, as before), we get $\mathcal{N}\subseteq \mathcal{N}^\mathbb{C}\cap \mathbb{P}_1=(\mathcal{N}^\mathbb{C})_\mathbb{P}$.

Suppose a morphism $n:P\rightarrow Q$ belongs to $(\mathcal{N}^\mathbb{C})_\mathbb{P}$. Then there is a commutative square
$$\xymatrix{ P'\ar[r]^-{n'}\ar@{->>}[d]_-{p} & Q'\ar@{->>}[d]^-{q} \\ P\ar[r]_-{n} & Q }$$ 
in $\mathbb{P}$, where $p$ and $q$ are regular epimorphisms in $\mathbb{C}$ and $n'\in\mathcal{N}$. By ``projectivity'' of $P$, the morphism $p$ has a right inverse $p'$. Then, $n=npp'=qn'p'\in\mathcal{N}$.

(b): Let $\mathcal{N}$ be an ideal in $\mathbb{P}$. Suppose $\mathcal{N}$ has weak $\mathcal{N}$-kernels. To show that $\mathbb{C}$ has  $\mathcal{N}^\mathbb{C}$-kernels, consider a morphism $f:B\rightarrow C$ in $\mathbb{C}$. Then, construct the diagram
$$\xymatrix{ P\ar[r]^-{k}\ar@{->>}[dd]_-{e} & Q\ar@{->>}[d]^-{q}\ar[r]^-{p_2q} & R\ar@{=}[d] \\  & B\times_C R\ar@{->>}[d]^-{p_1}\ar[r]^-{p_2} & R\ar@{->>}[d]^-{r} \\ A\ar@{ >->}[r]_-{m} & B\ar[r]_-{f} & C }$$
where the bottom right square is a pullback, with $r$ and $q$ arbitrary regular epimorphisms whose domains are in $\mathbb{P}$ (such morphisms exist since $\mathbb{P}$ is a projective cover of $\mathbb{C}$). Further, $k$ is a weak $\mathcal{N}$-kernel of $p_2q$ in $\mathbb{P}$ and $p_1 q k=me$ is a decomposition of $p_1q k$ via a regular epimorphism $e$ followed by a monomorphism $m$. Commutativity of the outer square whose vertical edges are regular epimorphisms and the top edge belongs to $\mathcal{N}$ shows that $fm\in\mathcal{N}^\mathbb{C}$. By a simple routine argument which we omit, $m$ is in fact an $\mathcal{N}^\mathbb{C}$-kernel of $f$. 

Suppose now that $\mathbb{C}$ has $\mathcal{N}^\mathbb{C}$-kernels. Consider a morphism $u:Q\rightarrow R$ in $\mathbb{P}$. Let $k:K\rightarrow Q$ be its $\mathcal{N}^\mathbb{C}$-kernel in $\mathbb{C}$. Then it is easy to see that for any regular epimorphism $p:P\rightarrow K$ with $P\in\mathbb{P}$, the composite $kp$ will be a weak $\mathcal{N}$-kernel of $u$ in $\mathbb{P}$.

(c): Now let $\mathcal{N}$ be an ideal in $\mathbb{C}$. If $\mathcal{N}$ admits kernels then weak kernels for $\mathcal{N}_\mathbb{P}$ can be constructed in the same way as right above: we take a morphism $u:Q\rightarrow R$ in $\mathbb{P}$ and consider its $\mathcal{N}$-kernel $k:K\rightarrow Q$ in $\mathbb{C}$; then, for any regular epimorphism $p:P\rightarrow K$ with $P\in\mathbb{P}$, the composite $kp$ will be a weak $\mathcal{N}$-kernel of $u$ in $\mathbb{P}$. Next, we show that when $\mathcal{N}$ admits kernels, we have $(\mathcal{N}_\mathbb{P})^\mathbb{C}\subseteq\mathcal{N}$. Consider a morphism $f:X\rightarrow Y$ which belongs to $(\mathcal{N}_\mathbb{P})^\mathbb{C}$. Then there is a commutative diagram
$$\xymatrix{ P\ar[r]^-{n}\ar@{->>}[d]_-{p} & Q\ar@{->>}[d]^-{q} \\ X\ar[r]_-{f} & Y }$$ where $p$ and $q$ are regular epimorphisms, $P$ and $Q$ are objects in $\mathbb{P}$, and $n\in\mathcal{N}$. Let $k$ be an $\mathcal{N}$-kernel of $f$. Then $p$ must factor through $k$ which forces $k$ to be a regular epimorphism and hence an isomorphism (since, being a kernel, it is already a monomorphism). Since an isomorphism is an $\mathcal{N}$-kernel of $f$, it follows that $f\in\mathcal{N}$. 

The proof of (d) is straightforward, while (e) is a consequence of (d) and (a). 
\end{proof}

The lemma above implies that the assignments $\mathcal{N}\mapsto\mathcal{N}^\mathbb{C}$  and $\mathcal{N}\mapsto\mathcal{N}_\mathbb{P}$ give rise to two Galois connections
$$\mathfrak{I}(\mathbb{P})\leftrightarrows\mathfrak{I}_\mathsf{s}(\mathbb{C})$$
$$\mathfrak{I}_\mathsf{wk}(\mathbb{P})\rightleftarrows\mathfrak{I}_\mathsf{k}(\mathbb{C})$$
where 
\begin{itemize}
\item $\mathfrak{I}_\mathsf{s}(\mathbb{C})$ denotes the ordered set of ideals in $\mathbb{C}$ admitting saturating regular epimorphisms;

\item $\mathfrak{I}_\mathsf{wk}(\mathbb{P})$ denotes the ordered set of ideals in $\mathbb{P}$ admitting weak kernels;

\item $\mathfrak{I}_\mathsf{k}(\mathbb{C})$ denotes the ordered set of ideals in $\mathbb{C}$ admitting kernels;

\item in each case, the top arrow indicates the direction of the left adjoint.
\end{itemize}
At the same time, they establish an isomorphism
$$\mathfrak{I}_\mathsf{wk}(\mathbb{P})\approx\mathfrak{I}_\mathsf{s}(\mathbb{C})\cap \mathfrak{I}_\mathsf{k}(\mathbb{C}).$$

Let $\mathbb{P}$ be a category with weak finite limits. Recall from \cite{CarVit98} that the regular completion of such a category is (up to  an equivalence of categories) a regular category $\mathbb{C}$ such that $\mathbb{P}$ is a subcategory of $\mathbb{C}$ and, moreover, the following conditions hold:
\begin{itemize}
\item $\mathbb{P}$ is a projective cover of $\mathbb{C}$ (see Section~\ref{SecA}).

\item For every object $C$ of $\mathbb{C}$ there is a monomorphism
$$\xymatrix{C\ar@{ >->}[r] & \prod_{i=1}^n P_i}$$
where each $P_i$ belongs to $\mathbb{P}$.
\end{itemize}

\begin{theorem}\label{The C} Let $(\mathbb{C},\mathcal{N})$ be a regular multi-pointed category with kernels, and let $\mathbb{P}$ be a projective cover of $\mathbb{C}$. Then:
\begin{itemize}
\item if $(\mathbb{C},\mathcal{N})$ is star-regular, then every reflexive graph in $(\mathbb{P},\mathcal{N}_\mathbb{P})$ satisfies $(\ast\pi_0)$;
\item the converse of this also holds, provided $\mathbb{C}$ is a regular completion of $\mathbb{P}$.
\end{itemize}
\end{theorem}

\begin{proof} Note that by Lemma~\ref{Lem A}(c), the multi-pointed category $(\mathbb{P},\mathcal{N}_\mathbb{P})$ has weak kernels. Moreover, as we showed in the proof of that lemma, a weak $\mathcal{N}_\mathbb{P}$-kernel of a morphism in $\mathbb{P}$ can be constructed by ``covering'' the domain of its  $\mathcal{N}$-kernel in $\mathbb{C}$ with a regular epimorphism whose domain is in $\mathbb{P}$. In particular, consider a diagram
$$\xymatrix{ P\ar@{->>}[r]^-{p} & K\ar[r]^-{k} & G_1\ar@/^4pt/[r]^-{d}\ar@/_4pt/[r]_-{c} & G_0\ar[l]|-{e}}$$
where $P,G_1,G_0$ are in $\mathbb{P}$, with $de=1_{G_0}=ce$, and $k$ is an $\mathcal{N}$-kernel of $d$ in $\mathbb{C}$ and $p$ is a regular epimorphism, so that $kp$ is a weak $\mathcal{N}_\mathbb{P}$-kernel of $d$ in $\mathbb{P}$. It is then easy to see that if the reflexive graph in this diagram satisfies $(\ast\pi_0)$ with respect to $\mathcal{N}$ in $\mathbb{C}$, then it also satisfies $(\ast\pi_0)$ with respect to $\mathcal{N}_\mathbb{P}$ in $\mathbb{P}$. Thus, the validity of the condition $(\ast\pi_0)$ for reflexive graphs in $(\mathbb{C},\mathcal{N})$ implies the validity of the same condition for reflexive graphs in $(\mathbb{P},\mathcal{N}_\mathbb{P})$. Together with Corollary~\ref{Cor A} this proves the first part of the theorem.

Conversely, assume that $\mathbb{C}$ is a regular completion of $\mathbb{P}$, and that reflexive graphs in $\mathbb{P}$ satisfy $(\ast\pi_0)$. Consider a diagram (\ref{Equ B}) in $\mathbb{C}$, with the middle-part being a reflexive graph while $k$ being an $\mathcal{N}$-kernel of $d$ in $\mathbb{C}$. Suppose $ldk=lck$. We want to show $ld=lc$. Extend this diagram to the diagram
$$
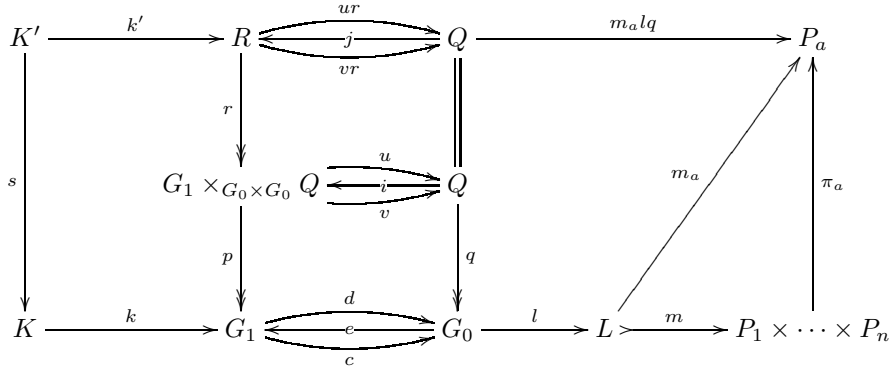
\xymatrix@=40pt{ K'\ar[r]^-{k'}\ar[dd]_-{s} & R\ar@{->>}[d]_-{r}\ar@/^7pt/[r]^-{ur}\ar@/_7pt/[r]_-{vr} & Q\ar@{=}[d]\ar[l]|-{j}\ar[rr]^-{m_alq} & & P_a \\ & G_1\times_{G_0\times G_0} Q\ar@{->>}[d]_-{p}\ar@/^7pt/[r]^-{u}\ar@/_7pt/[r]_-{v} & Q\ar@{->>}[d]^-{q}\ar[l]|-{i} & & \\  K\ar[r]^-{k} & G_1\ar@/^7pt/[r]^-{d}\ar@/_7pt/[r]_-{c} & G_0\ar[l]|-{e}\ar[r]^-{l} & L\ar@{ >->}[r]^-{m}\ar[uur]^-{m_a} & P_1\times\cdots\times P_n\ar[uu]_-{\pi_a} }$$
as follows: 
\begin{itemize}
\item $m$ is a monomorphism from $L$ into a product of objects $P_a$ from $\mathbb{P}$, while each $\pi_a$ is the $a$'th product projection (and $m_a=\pi_am$).

\item $q$ is any regular epimorphism whose domain $Q$ is in $\mathbb{P}$.

\item The indicated limit of the bottom middle square results in the projection $p$ also being a regular epimorphism (since it is obtained as a pullback of $q\times q$ which is a regular epimorphism because $q$ is). The morphism $i$ is the canonical morphism into this limit determined by the identities $pi=eq$ and $ui=1_Q=vi$.

\item $r$ is again any regular epimorphism whose domain $R$ is an object in $\mathbb{P}$.

\item By ``projectivity'' of $Q$, we get $j$ such that $rj=i$.

\item $k'$ is a weak $\mathcal{N}_\mathbb{P}$-kernel of $ur$ in $\mathbb{P}$. Since $dprk'=qurk'$ and $urk'\in\mathcal{N}_\mathbb{P}\subseteq\mathcal{N}$, it follows that $dprk'\in\mathcal{N}$ and so we get a morphism $s$ with $ks=prk'$. 
\end{itemize}
Now, since by our assumption $ldk=lck$, for each $a\in\{1,\dots,n\}$ we have
$$m_alqurk'=m_aldks=m_alcks=m_alqvrk'.$$
Since $m_alq$ is in $\mathbb{P}$, we can use the assumption on the top reflexive graph to conclude $m_alqur=m_alqvr$. Since $pr$ is an epimorphism, this gives $m_ald=m_alc$. Now, the $m_a$'s are jointly monomorphic, and so $ld=lc$ as desired.
\end{proof}

\begin{corollary}\label{Cor C}
Let $(\mathbb{P},\mathcal{N})$ be a multi-pointed category, where $\mathbb{P}$ has weak finite limits and weak $\mathcal{N}$-kernels. For the regular completion $\mathbb{C}$ of $\mathbb{P}$, the pair $(\mathbb{C},\mathcal{N}^\mathbb{C})$ is star-regular if and only if every reflexive graph in $\mathbb{P}$ satisfies $(\ast\pi_0)$.
\end{corollary}

\section{The pointed context}

A category $\mathbb{C}$ is said to be \emph{pointed} if there is an ideal $\mathcal{N}$ in $\mathbb{C}$, such that for any two objects $X$ and $Y$ in $\mathbb{C}$ there is exactly one morphism $X\rightarrow Y$ which belongs to $\mathcal{N}$. Such an ideal, when it exists, is necessarily unique. 

\begin{remark}\label{RemA}
Thus, for us a pointed category is simply a category enriched in the category of pointed sets. Note in particular that we do not require the existence of a zero object (unlike what is usually done in the literature), since its presence is not needed for our purposes.
\end{remark}

It is not difficult to show that if $\mathbb{P}$ is a projective cover of $\mathbb{C}$, then $\mathbb{P}$ is a pointed category if and only if so is $\mathbb{C}$, and moreover, if $\mathcal{N}$ denotes the ideal of null morphisms in $\mathbb{P}$ and  $\mathcal{M}$ denotes the ideal of null morphisms in $\mathbb{C}$, then we have $\mathcal{N}^\mathbb{C}=\mathcal{M}$ and $\mathcal{M}_\mathbb{P}=\mathcal{N}$. Thus, in the pointed case Corollary~\ref{Cor C} becomes:   

\begin{corollary}\label{Cor B}
Let $\mathbb{P}$ be a pointed category, where $\mathbb{P}$ has weak finite limits. The regular completion $\mathbb{C}$ of $\mathbb{P}$ is a normal category if and only if every reflexive graph in $\mathbb{P}$ satisfies $(\ast\pi_0)$ for the class $\mathcal{N}$ of null morphisms in $\mathbb{P}$.
\end{corollary}

Note that by Theorem~\ref{The C}, the `only if' part in the above result holds even when $\mathbb{P}$ is merely a projective cover of $\mathbb{C}$. We now show that the `if' part might not hold in this case, i.e.~it may be possible for reflexive graphs to satisfy $(\ast\pi_0)$ in a projective cover of a pointed regular category which is not normal.
\begin{example}
Consider a set of operators, $\Omega=\{d,0\}$, where $d$ is a binary operator and $0$ is a nullary operator. Let $E$ be the set of identities, consisting of the single identity $d(x,x)=0$. Take $\mathbb{C}$ to be the variety of $(\Omega,E)$-algebras, and let $\mathbb{C}'$ be the quasi-variety of $\mathbb{C}$ consisting of those algebras which satisfy the quasi-identity
\begin{equation}\label{qi} d(x,y)=0\quad\Rightarrow\quad x=y.\end{equation}
Let $\mathbb{P}$ be the full subcategory of $\mathbb{C}$ consisting of the free $(\Omega,E)$-algebras. Then $\mathbb{P}$ is a projective cover of $\mathbb{C}$. We claim that $\mathbb{P}$ is also a projective cover of $\mathbb{C}'$. For this, it suffices to show that every free $(\Omega,E)$-algebra satisfies the quasi-identity (\ref{qi}). Indeed:
\begin{itemize}
\item Consider the free $\Omega$-algebra $\mathsf{Fr}_\Omega X$ over an alphabet $X$, and define an operation $d':\mathsf{Fr}_\Omega X\times \mathsf{Fr}_\Omega X\rightarrow \mathsf{Fr}_\Omega X$ as follows: for any two distinct words $w_1$ and $w_2$, set $d'(w_1,w_2)$ to be the word $dw_1w_2$; for each word $w_1$ set $d'(w_1,w_1)=0$. 

\item Consider the smallest subset $F$ of $\mathsf{Fr}_\Omega X$ which contains the alphabet $X$, the word $0$, and is closed under the operation $d'$. We shall write $d'$ also for the restriction of the operation $d'$ to the set $F$.   

\item It is easy to see that $(F,d',0)$ is an $(\Omega,E)$-algebra, and, moreover, it satisfies the quasi-identity (\ref{qi}).  

\item This algebra is freely generated by the alphabet $X$. Indeed, we already know that it is generated by $X$. To show that it is freely generated by $X$, consider any map $f:X\rightarrow A$ from $X$ to an $(\Omega,E)$-algebra $A$, and let $f':\mathsf{Fr}_\Omega X\rightarrow A$ be the induced homomorphism of $\Omega$-algebras. Then the restriction $f''$ of $f'$ on $F$ is a homomorphism $f'':(F,d',0)\rightarrow A$ since for $w_1,w_2\in F$, where $w_1\neq w_2$,
$$f''(d(w_1,w_2))=f'(dw_1w_2)=d(f'(w_1),f'(w_2)),$$
and when $w_1=w_2$,
$$f''(d(w_1,w_2))=f''(0)=f'(0)=0.$$
At the same time, for any $x\in X$ we have $f''(x)=f'(x)=f(x)$.     
\end{itemize}  
Now, $\mathbb{C}'$ is a normal category (see~\cite{Fic68}). The first part of Theorem~\ref{The C} applied to $\mathbb{P}$ as the projective cover of $\mathbb{C}$ gives that reflexive graphs in $\mathbb{P}$ satisfy $(\ast\pi_0)$ (for the ideal of null morphisms in the pointed category $\mathbb{P}$). However, $\mathbb{C}$ is not a normal category. 
\end{example}
The example above also shows that normality somehow behaves differently with respect to regular completions than it does with respect to exact completions \cite{CarVit98}. Indeed, it shows that there exists a pointed category with weak finite limits, such that reflexive graphs in $\mathbb{P}$ satisfy $(\ast\pi_0)$, and at the same time the exact completion of $\mathbb{P}$ is not a normal category. On the other hand, by Corollary~\ref{Cor B}, the regular completion of such $\mathbb{P}$ is always normal. 

Similar questions were studied in \cite{Gra02} with respect to protomodular categories \cite{Bou91} in the place of normal categories. It is interesting to observe that, unlike normality, protomodularity has the same behaviour with respect to the regular and the exact completions. This is due to the fact that the condition used in \cite{Gra02} for $\mathbb{P}$ implies protomodularity of $\mathbb{C}$ already when $\mathbb{P}$ is merely a projective cover of $\mathbb{C}$. Note, however, that unlike in the present case, the condition used in \cite{Gra02} for $\mathbb{P}$ does not imply the same condition on $\mathbb{C}$. Moreover, protomodularity of $\mathbb{P}$ does not imply, in general, protomodularity of its regular completion. For instance, the regular completion of the protomodular category of groups is not a protomodular category (while it is still normal by Corollary~\ref{Cor B}). 

There is an analogous cleavage between normal categories and regular Mal'tsev categories \cite{CarLamPed90}, as it follows from the earlier investigation of the projective covers of regular Mal'tsev categories, due to J.~Rosick\'y and E.~M.~Vitale \cite{RosVit01}. The results obtained in \cite{GraRod12} show that the cases of unital and strongly unital categories \cite{Bou96}, and of subtractive categories \cite{ZJan05}, are also similar to that of protomodular categories. 

In a way, this shows that the behaviour of normality is not ``normal'', if we compare it to the behaviour of other similar exactness properties. It would be interesting to understand this phenomenon better, as well as to explore which other exactness properties studied in categorical algebra have similar unusual behaviour.

\end{document}